\documentclass[12pt,reqno]{article}

\usepackage[usenames]{color}
\usepackage{amssymb}
\usepackage{graphicx}
\usepackage{amscd}

\usepackage[colorlinks=true,
linkcolor=webgreen,
filecolor=webbrown,
citecolor=webgreen]{hyperref}

\definecolor{webgreen}{rgb}{0,.5,0}
\definecolor{webbrown}{rgb}{.6,0,0}

\usepackage{color}
\usepackage{fullpage}
\usepackage{float}

\usepackage{graphics,amsmath,amssymb}
\usepackage{amsthm}
\usepackage{amsfonts}
\usepackage{latexsym}
\usepackage{epsf}

\usepackage{verbatim}
\usepackage{enumerate}

\usepackage{tikz}
\usepackage[blocks]{authblk}

\theoremstyle{plain}
\newtheorem{theorem}{Theorem}
\newtheorem{corollary}[theorem]{Corollary}

\newtheorem{proposition}[theorem]{Proposition}
\newtheorem{conjecture}[theorem]{Conjecture}

\theoremstyle{definition}

\newtheorem{example}{Example}

\theoremstyle{remark}

\begin{document}

\title{Maximum Number of Quads}
\author{Nikhil Byrapuram}
\author{Hwiseo (Irene) Choi}
\author{Adam Ge}
\author{Selena Ge}
\author{Sylvia Zia Lee}
\author{Evin Liang}
\author{Rajarshi Mandal}
\author{Aika Oki}
\author{Daniel Wu}
\author{Michael Yang}
\affil{PRIMES STEP}
\author{Tanya Khovanova}
\affil{MIT}

\maketitle

\begin{abstract}
We study the maximum number of quads among $\ell$ cards from an EvenQuads deck of size $2^n$. This corresponds to enumerating quadruples of integers in the range $[0,\ell-1]$ such that their bitwise XOR is zero. In this paper, we conjecture a formula that calculates the maximum number of quads among $\ell$ cards.
\end{abstract}

\section{Introduction}

A notion of quad is a generalization of a notion of set from the game of SET. The game of SET is played with a deck of 81 cards \cite{MGGG}, where each card is characterized by 4 attributes:
\begin{itemize}
\item Number: 1, 2, or 3 symbols.
\item Color: green, red, or purple.
\item Shading: empty, striped, or solid.
\item Shape: oval, diamond, or squiggle.
\end{itemize}

A set is formed by three cards that are either all the same or all different in each attribute. The players try to find sets among given cards, and the fastest player wins.

In this paper, we are interested in a different game called EvenQuads, introduced by Rose and Perreira \cite{Rose}, which is a generalization of the SET game. The EvenQuads deck consists of $64 = 4^3$ cards with different objects. The cards have 3 attributes with 4 values each:
\begin{itemize}
\item Number: 1, 2, 3, or 4 symbols.
\item Color: red, green, yellow, or blue.
\item Shape: square, icosahedron, circle, or spiral.
\end{itemize}

A quad consists of four cards so that for each attribute, the values of the cards must be one of three cases: all the same, all different, or half and half. We can view cards as integers from 0 to 63 and quads as four numbers that bitwise XOR to 0. More about the basic properties of the game can be found in \cite{CragerEtAl,LAGames}. Latin, magic, semimagic, and other squares based on the quad notion can be found in \cite{QuadSquares}.

One of the most famous questions in the game of set is the question of the largest size of the no-set. Such sets are also called cap sets, and their sizes are an active area of research \cite{MGGG}. It is known that the largest possible number of cards in the standard deck not containing a set is 20. The same value for any dimension is described by sequence A090245$(n)$, defined as the maximum number of cards that would have no set in an $n$-attribute version of the SET card game. The maximum known value is for 6 attributes:
\[1,\ 2,\ 4,\ 9,\ 20,\ 45,\ 112.\]

A related question is the following. Given the deck size and the number of cards, how many sets are possible? In particular, we can ask what the maximum number of sets is among given cards. This question for the standard deck is discussed in \cite{Vinci}.

In this paper, we are interested in the corresponding question for the EvenQuads deck. Given the size of the deck and the number of cards, what is the maximum possible number of quads?

Section~\ref{sec:background} covers the background information.

In Section~\ref{sec:packed}, we define a packed set as a set of cards from a given deck that realizes the maximum possible number of quads. We provide examples for fewer than 9 cards. We also discuss the symmetries of the EvenQuads deck.

In Section~\ref{sec:complementary}, we discuss complementary sets: disjoints sets whose union is the whole deck. We prove a formula for the difference in the number of quads in such sets. Interestingly, this difference depends only on the number of cards in these sets. This means that if one set is packed, then its complementary set is also packed.

In Section~\ref{sec:conjectures}, we introduce a conjecture that leads to a claim that a set of $\ell$ cards numbered 0 through $\ell -1$ is packed. In particular, the number of quads in a packed set does not depend on the size of a deck the cards are chosen from, as long as the deck size is not smaller than the number of cards. Thus, we can view it as a function of $\ell$, which we denote as $Q(\ell)$.

In Section~\ref{sec:ncps}, based on our conjecture, we calculate $Q(\ell)$: the number of quads in a packed set of size $\ell$.

In Section~\ref{sec:partialsums}, we provide the formula for the sequence that is partial sums of $Q(\ell)$.

\section{Preliminaries}
\label{sec:background}

The EvenQuads deck consists of $64 = 4^3$ cards with different objects. The cards have 3 attributes with 4 values each:
\begin{itemize}
\item Number: 1, 2, 3, or 4 symbols.
\item Color: red, green, yellow, or blue.
\item Shape: square, icosahedron, circle, or spiral.
\end{itemize}

A \textit{quad} consists of four cards so that for each attribute, the values of the cards must be one of three cases: all the same, all different, or half and half. We can assume that each attribute takes values in the set $\{0,1,2,3\}$. Then, four cards form a quad if and only if the bitwise XOR of the values in each attribute is zero. This is equivalent to saying that each attribute takes values in $\mathbb{Z}_2^2$, and four cards form a quad if and only if for each attribute, the four vectors in $\mathbb{Z}_2^2$ sum to the zero vector.

Thus, we can view our cards as vectors in $\mathbb{Z}_2^6$. For generalizations, we can consider an EvenQuads deck of sizes $2^n$. The cards in such a deck correspond to vectors in $\mathbb{Z}_2^n$. Four vectors $\vec{a}$, $\vec{b}$, $\vec{c}$, and $\vec{d}$ in $\mathbb{Z}_2^n$ form a quad if and only if (see \cite{CragerEtAl})
\[\vec{a} + \vec{b} + \vec{c} + \vec{d} = \vec{0}.\]
Consider four vectors $\vec{a}$, $\vec{b}$, $\vec{c}$, and $\vec{d}$ forming a quad. By a translation, we can assume that $\vec{a}$ is the origin. Then, from the above equation we get $\vec{d} = - \vec{b} - \vec{c} = \vec{b} + \vec{c}$. This means that vector $\vec{d}$ is in the same plane as the origin and vectors $\vec{b}$ and $\vec{c}$. Thus, four cards form a quad if and only if their endpoints belong to the same plane in $\mathbb{Z}_2^n$.

In the rest of the paper, we often use numbers from 0 to $2^n-1$ inclusive to label the cards in the EvenQuads-$2^n$ deck. Four numbers form a quad if and only if their bitwise XOR is 0.

The following statements are left for the reader to check.
\begin{itemize}
\item In a given deck, any three cards can be completed into a quad.
\item It follows that two different quads can have at most two common cards.
\end{itemize}

We are interested in the number of possible quads given $\ell$ cards. Suppose that there exists a set of cards that produces $q$ quads with $\ell$ cards from an EvenQuads-$2^n$ deck. By adding more cards to the deck, we can conclude that any larger deck can also have $\ell$ cards with $q$ quads. Thus, it is enough to find the smallest such deck. Denote by $D(\ell,q)$ the smallest deck size $2^n$ such that there exist $\ell$ cards from the EvenQuads-$2^n$ deck that form exactly $q$ quads. If such a deck does not exist, we define $D(\ell,q)$ as $\infty$.

Values of $D(\ell,q)$ for $\ell \leq 8$ were calculated in \cite{ECC} and shown in Table~\ref{table:possibleNumQuads}. The first column represents the number of cards, and the first row represents the number of quads. If the number of cards is $\ell$ and the number of quads is $q$, the entry corresponding to $\ell$ and $q$ shows $D(\ell,q)$. We replaced the infinity sign with an empty entry so as not to clutter the table.

\begin{table}[ht!]
\begin{center}
\begin{tabular}{|c|c|c|c|c|c|c|c|c|c|c|c|c|c|c|c|}
\hline
\# of cards$\backslash$quads	& 0	& 1     & 2    & 3   & 4 & 5   & 6 & 7 & 8 & 9 & 10 & 11 & 12 & 13 & 14        \\ \hline
1 		& 1	&              &              &              &   &              &   &   &  &  &  &  &  &  &           \\ \hline
2 		& 2	&              &              &              &   &              &   &   &  &  &  &  &  &  &              \\ \hline
3 		& 4 	&              &              &              &   &              &   &    &  &  &  &  &  &  &             \\ \hline
4 		& 8 	& 4 		&              &              &   &              &   &    &  &  &  &  &  &  &             \\ \hline
5 		& 16	& 8 		&              &              &   &              &   &   &  &  &  &  &  &  &              \\ \hline
6 		& 16	& 16 		&              & 8 		&   &              &   &    &  &  &  &  &  &  &             \\ \hline
7 		& 32	& 32		& 16		& 16		&   & 		 &   & 8  &  &  &  &  &  &  &  \\ \hline
8 		& 64	& 32		& 32		& 32		&   & 16     &  16   & 16  &  &  &  &  &  &  & 8 \\ \hline
\end{tabular}
\end{center}
\caption{$D(\ell,q)$ for $\ell \leq 8$.}
\label{table:possibleNumQuads}
\end{table}

\subsection{Symmetries}

There is a natural approach to symmetries of quads using linear algebra. Our space is an affine space (the same as a vector space, but without choosing the origin). The symmetries of the space are affine transformations. Affine transformations consist of parallel translations and linear transformations. A linear transformation is defined by an invertible matrix.

We can view an affine transformation as a pair $(M,\vec{t})$, where $M$ is an invertible $n$-by-$n$ matrix over the field $\mathbb{F}_2$ and $\vec{t}$ is a translation vector in $\mathbb{F}_2^n$. The pair acts on vector $\vec{a}$ as $M\vec{a} + \vec{t}$.

An affine transformation preserves quads. Indeed if $\vec{a_i}$, for $1 \leq i \leq 4$ form a quad then the transformation produces four vectors $\vec{b_i} = M\vec{a_i} + \vec{t}$. We have
\[\sum_{i=1}^4\vec{b_i} = \sum_{i=1}^4(M\vec{a_i} + \vec{t}) = \sum_{i=1}^4 M\vec{a_i} + \sum_{i=1}^4\vec{t} = M\sum_{i=1}^4\vec{b_i} + 4\vec{t} = 0,\]
thus vectors $\vec{b_i}$ form a quad.

\section{Maximum numbers of quads: packed sets}
\label{sec:packed}

What is the maximum number of quads in a selection of $\ell$ cards? We define an $\ell$-packed set for a given value $\ell$ to be a set of $\ell$ cards from some EvenQuads deck such that the number of quads within these cards is maximized. In other words, given the number of cards $\ell$, we are looking at the largest index $q$, such that $D(\ell,q)$ is not infinity. We denote this number as $Q(\ell)$. From our later conjecture, it follows that $Q(\ell)$ does not depend on the size of the deck the cards are chosen from.

The maximum number of quads for small values of $n$ was calculated in \cite{ECC} and can be found in Table~\ref{table:possibleNumQuads}. For a given number of cards $\ell$, the value of $Q(\ell)$ is the largest possible column number such that Table~\ref{table:possibleNumQuads} has an entry in this column in row $\ell$. Here are the values of $Q(\ell)$ starting with $\ell = 1$.
\[0,\ 0,\ 0,\ 1,\ 1,\ 3,\ 7,\ 14,\ \ldots.\]

Let us define a \textit{complete} set as a set of cards such that any three cards from that set can be completed into a quad using a fourth card from that set. It means the fourth card belongs to a plane that passes through the given three cards.

\begin{proposition}
If a set of $\ell$ cards is complete, then $\ell$ is a power of 2. Moreover, for any $\ell$ that is a power of 2, there exists a set of $\ell$ cards that is complete.
\end{proposition}

\begin{proof}
Consider the cards as elements of $\mathbb Z_2^k$, with numbering chosen so that $\vec{0}$ is one of the $\ell$ cards. Then, since $\vec{0}$, $\vec{a}$, and $\vec{b}$ can be completed into a quad using a card from the given $\ell$ cards for any cards $\vec{a}$ and $\vec{b}$, the set of $\ell$ cards is closed under addition. Thus, the set of cards forms a vector space over $\mathbb Z_2$, so the number of cards is a power of 2.

On the other hand, any subspace of a vector space is closed under addition. That means, given 3 cards from this subspace, the fourth card that completes the quad belongs to the same subspace.
\end{proof}

We can estimate the upper bound for the number of quads in the $\ell$ cards set.

\begin{theorem}
\label{thm:bound}
Given $\ell$ cards, they contain at most $\frac{\binom \ell3}{4}$ quads. Equality occurs if and only if the set is complete.
\end{theorem}

\begin{proof}
Take any three cards out of those $\ell$ cards. Some triplets can be completed into quads using only cards from those $\ell$ cards. There are at most $\binom \ell3$ such triplets, and there will be exactly $\binom \ell3$ such triplets if and only if every triplet can be completed into a quad using a card from those $\ell$ cards. If a triplet can be completed into a quad using cards from those $\ell$ cards, then there are four triplets that yield that quad. So there are at most $\frac{\binom \ell3}{4}$ quads using cards chosen from $\ell$ cards.
\end{proof}

The numbers $\binom \ell3$ are also called tetrahedral numbers, see sequence A000292: 
\[0,\ 0,\ 1,\ 4,\ 10,\ 20,\ 35,\ 56,\ 84,\ 120,\ 165,\ 220,\ 286,\ \ldots.\]
Our bound is 
\[\left\lfloor \frac{\binom \ell3}{4} \right\rfloor,\]
which is a shifted sequence A011842 (our index starts at 1):
\[0,\ 0,\ 0,\ 1,\ 2,\ 5,\ 8,\ 14,\ 21,\ 30,\ 41,\ 55,\ 71,\ 91,\ 113,\ 140,\ 170,\ 204,\ 242,\ \ldots.\]

\begin{example}
If $\ell = 2^k$, where $k$ is a non-negative integer, then the maximum number of quads is $\frac{\binom \ell3}{4}$. In other words, $Q(2^k) = \frac{\binom \ell3}{4}$.The corresponding sequence as a function of $k$ for $k > 1$ is sequence A016290:
\[1,\ 14,\ 140,\ 1240,\ 10416,\ 85344,\ 690880,\ 5559680,\ \ldots.\]
\end{example}

We can extend this example to numbers one greater than a power of 2.

\begin{proposition}
\label{prop:afterpower2}
If $\ell = 2^k+1$, where $k$ is a non-negative integer, then the maximum number of quads is achieved when our points form a $k$-dimensional hyperplane with an extra point. It follows that $Q(2^k+1) = Q(2^k)$.
\end{proposition}

\begin{proof}
Suppose this is not the case, and another selection $S$ of $\ell$ points has more quads. Let point $x \in S$ be a point that forms $y$ quads with other points. Denote the other $2^k$ cards as $S'$. If $y=0$, and we remove point $x$, we get that $\ell -1$ points form more quads than the proven maximum for $2^k$ points, which is impossible. Thus, point $x$ forms at least one quad with other points.

Let $a$, $b$, and $c$ be three points in $S'$ that $x$ forms a quad with. This means that $a$, $b$, and $c$ can not form a quad with another point in $S'$. Suppose $d$, $e$, and $f$ are some other three points in $S'$ that $x$ forms a quad with. Two triples $a$, $b$, $c$ and $d$, $e$, $f$ that form distinct quads with $x$ can share at most one point in common (as three points determine a quad). Thus, two triples point to different quads in $S'$. This means that the maximum number of quads within $S'$ is $Q(2^k)-y$. This implies that the maximum number of quads within the $2^k+1$ points is $Q(2^k)$, which contradicts the assumption.

Thus, for $\ell = 2^k+1$, we have $Q(\ell) \leq Q(\ell-1) = \frac{\binom {\ell-1}{3}}{4}$. The value $Q(\ell-1)$ is achievable within $2^k+1$ points when we have a $k$-dimensional hyperplane and an additional card that does not form a quad with other cards. Thus, for $\ell = 2^k+1$, we have $Q(\ell) = Q(\ell-1) = \frac{\binom {\ell-1}{3}}{4}$.
\end{proof}

\section{Complementary sets}
\label{sec:complementary}

Given a finite EvenQuads deck, two sets are called \textit{complementary} if they are disjoint, and their union is the whole deck. The following theorem provides the formula that connects the number of quads in two complementary sets. We denote the number of cards in a set $S$ as $|S|$.

\begin{theorem}
\label{thm:complementarysets}
Suppose $S$ and $T$ are two complementary sets in a finite EvenQuads deck. Suppose $s$ is the number of quads in $S$, and $t$ is the number of quads in $T$. Then
\begin{multline*}
s-t=\frac{1}{4} \left[ {|S| \choose 3} - {|T| \choose 3} \right] + \frac{1}{12} \left[ |S| {|T| \choose 2} - |T| {|S| \choose 2} \right] \\
= \frac{|S| - |T|}{24}(|S|^2 + |T|^2 - 3|S| - 3|T| +2).
\end{multline*}
\end{theorem}

\begin{proof}
Let $q_i$ be the number of quads with exactly $i$ cards in $S$, for $1 \leq i \leq 3$. Define an \textit{ordered quad} as an ordered 4-tuple of cards that form a quad.

First, we count the number of ordered quads whose first three elements are from $T$ in two different ways. One way is to choose three elements from $T$ in $6{|T| \choose 3}$ ways and then add the last card that makes them a quad. Another way is that we can take a quad with all 4 cards in $T$ and order it in 24 ways, or a quad with exactly one card in $S$ and order it in 6 ways (this covers all of the ordered quads that have their first three elements from $T$). Thus, $6{|T| \choose 3} = 24t + 6q_1$, or
\[{|T| \choose 3} = 4t + q_1.\]

Next, we count, in two ways, the number of ordered quads whose first two elements are from $T$ and whose third element is from $S$. One way is to pick the first three elements in $2{|T| \choose 2}|S|$ ways and then complete them into a quad. Another way is to take a quad with one card from $S$ and three cards from $T$ and order it in 6 ways, or a quad with two cards from $S$ and two cards from $T$ and order it in one of the 4 ways. Thus, $2{|T| \choose 2}|S| = 6q_1 + 4q_2$, or
\[|S|{|T| \choose 2} = 3q_1 + 2q_2.\]

Similarly, by swapping the roles of $S$ and $T$, we get
\[|T|{|S| \choose 2} = 2q_2 + 3q_3 \quad \textrm{and} \quad{|S| \choose 3} = q_3 + 4s.\]

From the first and the last equations, we get
\[{|S| \choose 3} - {|T| \choose 3}  = 4(s-t) + q_3 -q_1.\]
From the other two equations, we get
\[q_3 -q_1 = |T|{|S| \choose 2} - |S|{|T| \choose 2}.\]
Combining this, we get the desired result.
\end{proof}

The theorem helps us estimate the number of quads. If we know a possible number of quads among $\ell$ cards in a given deck of size $2^k$, then we know a possible number of quads among $2^k - \ell$ cards in the same deck. We get a very useful corollary from this theorem.

\begin{corollary}
\label{cor:packed_complements}
If we have a packed set $S$ in a given EvenQuads deck, then its complementary set is also packed.
\end{corollary}

\begin{proof}
Consider the complement $T$ to set $S$, and denote the number of quads in $T$ and $S$ as $t$ and $s$, respectively. The difference $s-t$ only depends on the sizes of sets $S$ and $T$. Thus, the number of quads is maximized/minimized in $S$ and $T$ simultaneously.
\end{proof}

Note that the studies of packed sets for SET and Quads are qualitatively different. In SET, since 3 is odd, the complementary set formula gives the total number of sets for a given set of cards and its complement, while in EvenQuads, the complementary set formula gives the difference.

\begin{example}
We know that the maximum number of quads given 0--8 cards is 0, 0, 0, 0, 1, 1, 3, 7, and 14, respectively. All these cases are realizable in a standard deck. Thus, for 56--64 cards, which form complements to the above sets in the standard deck, the maximum number of quads are 10416, 9765, 9145, 8555, 7995, 7462, 6958, 6482, and 6034, respectively.
\end{example}

\subsection{Function $F$}

We introduce function $F(\ell)$. We later conjecture that $F$ equals the maximum possible number of quads among $\ell$ cards. However, we define $F$ using properties of complementary sets and induction.

We assume that $F(0)=0$. Suppose $F(i)$ is defined for all integers in the range $[0\ldots 2^{k-1}]$, we now define it for $p \in (2^{k-1}, 2^k]$. Let $q = 2^k - p$, then $q \le 2^{k-1}$. This means $F(q)$ is already defined. Now, we define $F(p)$ as follows:
\[F(p) =F(q)+\frac{3\binom{p}{3}-q\binom{p}{2}+p\binom{q}{2}-3\binom{q}{3}}{12}.\]

\begin{example}
If $p=2^k$, then $q=0$, and we define $F(2^k)$ as 
\[F(0)+\frac{3\binom{2^k}{3}-0 \cdot \binom{2^k}{2}+2^k \cdot \binom{0}{2}-3\binom{0}{3}}{12} = \frac{\binom{2^k}{3}}{4}.\]
\end{example}

\section{Conjectures}
\label{sec:conjectures}

Before calculating the maximum number of quads among $\ell$ cards, $Q(\ell)$, we have a conjecture.

\begin{conjecture}
\label{conj:reduction}
If a set of $\ell$ cards is not contained in a complete set of $2^{\lceil\log_2 \ell\rceil}$ cards, then it has at most as many quads as an $(\ell-1)$-packed set.
\end{conjecture}

Now, we have the following proposition.

\begin{proposition}
\label{prop:contains}
If Conjecture~\ref{conj:reduction} holds, then there exists an $\ell$-packed set that is contained in a complete set of cards of dimension $2^{\lceil\log_2 \ell\rceil}$.
\end{proposition}

\begin{proof}
Suppose, on the contrary, there does not exist an $\ell$-packed set that is contained in a complete set of cards, $S_c$, of dimension $2^{\lceil\log_2 \ell\rceil}$. Consider an $\ell$-packed set $S$, which, by our assumption, is not contained in $S_c$. Then, from Conjecture~\ref{conj:reduction}, set $S$ has at most the number of quads as an $(\ell-1)$-packed set. We can keep applying Conjecture~\ref{conj:reduction} until we get an $\ell'$-packed set (where $\ell'<\ell$), $S'$, that is contained in $S_c$ and has at least as many quads as $S$. We can then add, arbitrarily, $\ell-\ell'$ cards from $S_c \setminus S'$ into $S'$ to get a set $S''$ of $\ell$ cards that is contained in $S_c$ and that has at least as many quads as $S$, which is a contradiction.
\end{proof}

\begin{theorem}
\label{thm:mapping}
If Conjecture~\ref{conj:reduction} holds and $S$ is an $\ell$-packed set, then there is a symmetry of quads mapping $S$ to $\{0,1,\ldots,\ell-1\}$.
\end{theorem}

\begin{proof}
For $\ell < 4$, any set of $\ell$ cards does not contain a quad. For $\ell = 4$, a set of cards contains 1 quad if and only if the set belongs to the same plane. We proceed by induction. Suppose the theorem is true for $\ell \leq 2^{k-1}$. We now show that it has to be true for $2^{k-1} < \ell \leq 2^k$.

By Conjecture~\ref{conj:reduction}, we can assume that our packed set $S$ must be contained in a complete set $U$ of $2^k$ cards. All complete sets are equivalent, so we can assume that $U$ is the set $\{0,1,\ldots,2^k-1\}$. By studying complementary sets, we know that our set is packed if and only if its complement $U \setminus S$ is packed. 

As the size of the set $U \setminus S$ does not exceed $2^{k-1}$, by the inductive hypothesis we can assume that $U \setminus S$ are the carsds $\{0,1,\ldots,2^k- \ell -1\}$ in $U$. Thus, the set $S$ consists of cards $\{2^k-\ell,2^k -\ell +1,\ldots,2^k-1\}$. Finally, we apply the symmetry of bitwise xoring with $2^k-1$, mapping $S$ onto $\{0,1,\ldots,\ell-1\}$, completing the induction.
\end{proof}

In particular, this means that a notion of a packed set does not depend on the deck size. The number of quads in a packed set of size $\ell$ is the same for any deck that contains at least $\ell$ cards.

\begin{example}
If $\ell = 2^k-1$, where $k > 1$, then the maximum number of quads is achieved, according to Proposition~\ref{prop:contains} when our points form a $k$-dimensional hyperplane without a point. Thus, from all quads in a hyperplane, we need to subtract quads that do not contain a given point. The number of quads in the hyperplane is $\frac{{\ell+1 \choose 3}}{4}$. Consider a point $x$ in our hyperplane. We count the number of quads containing $x$. In our hyperplane, there are ${\ell \choose 2}$ pairs of cards that form triples with $x$; each of these triples is completed to a quad. However, each quad is counted three times. Thus, the number of quads containing $x$ is $\frac{{\ell \choose 2}}{3}$. The total number of quads in our set is 
\[\frac{{\ell+1 \choose 3}}{4} - \frac{{\ell \choose 2}}{3}.\]
This number can be simplified to $\frac{\ell(\ell-1)(\ell-3)}{24}$.
\end{example} 

For example, when $\ell=7$, this number is 7; when $\ell=15$, this number is 105. Here are the first ten terms for the corresponding sequence $a(n) = \frac{1}{24}(2^k - 1) (2^k - 2) (2^k - 4) = \frac{{2^k \choose 3}}{4} - \frac{{2^k-1 \choose 2}}{3}$ (starting from index 1):
\[0,\ 0,\ 7,\ 105,\ 1085,\ 9765,\ 82677,\ 680085,\ 5516245,\ 44434005.\]
Given that powers of 2 can only have remainders 1, 2, or 4 modulo 7, we can conclude that $a(n)$ is divisible by 7. After dividing we get sequence  A006096 in the OEIS \cite{OEIS}:
\[0,\ 0,\ 1,\ 15,\ 155,\ 1395,\ 11811,\ 97155,\ 788035,\ 6347715,\ \ldots.\]

\section{The number of cards in a packed set}
\label{sec:ncps}

Now we show that $Q(\ell) = F(\ell)$.

\begin{theorem}
Assuming Conjecture~\ref{conj:reduction}, the number of quads in an $\ell$-packed set is $F(\ell)$.
\end{theorem}

\begin{proof}
By Theorem~\ref{thm:mapping} we can assume that our $\ell$-packed set consists of cards $\{0,1,\ldots, \ell-1\}$.

We know that $Q(0) = Q(1) = F(0) = F(1) = 0$. Suppose the theorem is true for $p \le 2^{k-1}$. We now use induction on $k$. We prove it separately for $2^{k-1} < p < 2^k$, and $p = 2^k$.

If $p = 2^k$, then a $p$-packed set is a complete set with $\frac{\binom{2^k}{3}}{4}$ quads. We also know that $F(2^k) = \frac{\binom{2^k}{3}}{4}$. Thus, $Q(2^k) = F(2^k)$.

If $2^{k-1} < p < 2^k$, then we introduce $q = 2^k - p$, and, therefore, $q < 2^{k-1}$. By the induction hypothesis $Q(q) = F(q)$. We also know by Theorem~\ref{thm:mapping}, that the set of cards $\{0,1,\ldots, q-1\}$ is $q$-packed. As the number of quads in this set is maximized, the same is true by Corollary~\ref{cor:packed_complements} for its complement $\{q,q+1,\ldots, 2^k\}$ in the deck of size $2^k$. Thus the set $\{q,q+1,\ldots, 2^k\}$ is packed, and by symmetry the set $\{0,1,\ldots, p-1\}$ is also packed. Hence, by Theorem~\ref{thm:complementarysets} the number of quads in a $p$-packed set is 
\[Q(p) = Q(q)+\frac{3\binom{p}{3}-q\binom{p}{2}+p\binom{q}{2}-3\binom{q}{3}}{12} = F(q)+\frac{3\binom{p}{3}-q\binom{p}{2}+p\binom{q}{2}-3\binom{q}{3}}{12} = F(p).\]
\end{proof}

Table~\ref{table:Q} describes the values $Q(\ell)$ calculated above for the number of cards up to 63.

\begin{table}[ht!]
\begin{center}
\begin{tabular}{c|c}
0&0\\
1&0\\
2&0\\
3&0\\
4&1\\
5&1\\
6&3\\
7&7\\
8&14\\
9&14\\
10&18\\
11&26\\
12&39\\
13&55\\
14&77\\
15&105
\end{tabular}\quad
\begin{tabular}{c|c}
16&140\\
17&140\\
18&148\\
19&164\\
20&189\\
21&221\\
22&263\\
23&315\\
24&378\\
25&442\\
26&518\\
27&606\\
28&707\\
29&819\\
30&945\\
31&1085
\end{tabular}\quad
\begin{tabular}{c|c}
32&1240\\
33&1240\\
34&1256\\
35&1288\\
36&1337\\
37&1401\\
38&1483\\
39&1583\\
40&1702\\
41&1830\\
42&1978\\
43&2146\\
44&2335\\
45&2543\\
46&2773\\
47&3025
\end{tabular}\quad
\begin{tabular}{c|c}
48&3300\\
49&3556\\
50&3836\\
51&4140\\
52&4469\\
53&4821\\
54&5199\\
55&5603\\
56&6034\\
57&6482\\
58&6958\\
59&7462\\
60&7995\\
61&8555\\
62&9145\\
63&9765
\end{tabular}
\end{center}
\caption{Values of $Q(\ell)$: the conjectured maximum number of quads for $\ell$ cards.}
\label{table:Q}
\end{table}

\section{Partial sums of $Q(n)$}
\label{sec:partialsums}

Consider sequence $P(n) = Q(n+1) - Q(n)$, the sequence of partial differences of $Q(n)$. Starting from index zero, this sequence is the following:
\[0,\ 0,\ 0,\ 1,\ 0,\ 2,\ 4,\ 7,\ 0,\ 4,\ 8,\ 13,\ 16,\ 22,\ 28,\ 35,\ 0,\ 8,\ 16,\ 25,\ 32,\ 42,\ 52,\ 63, \ldots .\]
Correspondingly, we have that $Q(n)$ is the sequence of partial sums of $P(n)$:
\[Q(n+1) = Q(n) + P(n).\]

Plugging in the values into the OEIS \cite{OEIS}, we see that the same numbers appear as sequence A213673, defined as
\[\text{A213673}(n)=\frac{n^2-\text{A000695}(n)}{4}.\]

Sequence A000695 is called the Moser-de Bruijn sequence. It is an increasing sequence of sums of distinct powers of 4. Equivalently, these are numbers whose base-4 representation consists of only digits 0 and 1. In other words, these are numbers whose binary representations are nonzero only in even positions. The sequence starts as
\[0,\ 1,\ 4,\ 5,\ 16,\ 17,\ 20,\ 21,\ 64,\ 65,\ 68,\ 69,\ \ldots \]
We can also say that to calculate A000695$(n)$, we need to write $n$ in binary and evaluate the result in base 4.
 
We found a match for $P(n)$ in the OEIS database. Now, we need to prove that the sequences coincide.

\begin{theorem}
Sequence $P(n)$ is sequence A213673$(n)$.
\end{theorem}

\begin{proof}
By Proposition~\ref{prop:afterpower2}, we have $Q(2^k)=Q(2^k+1)$ and therefore, $P(2^k)=0$. By one of the definitions of sequence A000695$(n)$, we have that A000695$(2^k)=2^{2k}$ and therefore A213673$(2^k)=0 = P(2^k)$.

Now, we proceed by induction. We assume that the statement is proven for any integer not greater than $2^k$. Let $m$ be an integer in the interval $[2^k+1,2^{k+1}-1]$. In particular, $m$ is not a power of 2. Let us denote $2^{k+1}$ by $N$. We can use the formula for complementary sets to get
\begin{equation}
\label{eq:m}
Q(m)=Q(N-m)+\frac{3\binom{m}{3}-(N-m)\binom{m}{2}+m\binom{N-m}{2}-3\binom{N-m}{3}}{12}.
\end{equation}
Since $N$ is also the smallest power of 2 that is at least $m+1$, we can use the formula for complementary sets for $m+1$ to get
\begin{equation}
\label{eq:m+1}
Q(m+1)=Q(N-m-1)+\frac{3\binom{m+1}{3}-(N-m-1)\binom{m+1}{2}+(m+1)\binom{N-m-1}{2}-3\binom{N-m-1}{3}}{12}.
\end{equation}

Subtracting Eq.~\ref{eq:m} from Eq.~\ref{eq:m+1} we get
\begin{align}
P(m)&=Q(m+1)-Q(m) \notag \\
&=Q(N-m-1)-Q(N-m)+\frac{1}{6}\left(3m^2-3mN+3m+N^2-3N+2\right) \notag \\
&=\frac{1}{6}\left(3m^2-3mN+3m+N^2-3N+2\right)-P(N-m-1). \notag
\end{align}

Given that $P(N-m-1) < 2^k$, by our assumption $P(N-m-1) = \text{A213673}(n)$. Thus
\begin{equation}
\label{eq:3}
P(m) = \frac{1}{6}\left(3m^2-3mN+3m+N^2-3N+2\right) - \textrm{A213673}(N-m-1).
\end{equation}
The binary representation of $N-m-1$ is obtained by complementing that of $m$. This means that
\[\textrm{A000695}(m) + \textrm{A000695}(N-m-1)=11\ldots 11_4\]
with $k+1$ ones in total. On the other hand, we have
\[11\ldots 11_4 = \frac{4^{k+1}-1}3=\frac{N^2-1}{3}.\] Using the definition of A213673, we obtain A000695$(n)=n^2 - 4\textrm{A213673}(n)$. Therefore, plugging this into Eq.~\ref{eq:3}, we get
\[
m^2 - 4\textrm{A213673}(m)+(N-m-1)^2 - 4\textrm{A213673}(N-m-1)=\frac{N^2-1}{3},
\]
which is equivalent to
\[
\textrm{A213673}(m) + \textrm{A213673}(N-m-1) = \frac{1}{6}\left(3m^2-3mN+3m+N^2-3N+2\right).
\]
Combining this with Eq.~\ref{eq:3}, we get $P(n) = \textrm{A213673}(n)$.
\end{proof}

\section{Acknowledgments}

We are grateful to the MIT PRIMES STEP program and its director, Slava Gerovitch, for allowing us the opportunity to do this research. We also thank  Pr.~Yufei Zhao for consulting us.

\end{document}